\documentclass[11pt]{article}
\pdfoutput=1
\usepackage[ascii]{inputenc}

\usepackage[T1]{fontenc}
\usepackage[english]{babel}
\usepackage{amsmath,amssymb,amsfonts,textcomp,epsfig,amsthm}
\usepackage[pdftex,colorlinks=true,citecolor=blue]{hyperref}
\usepackage{color}
\usepackage{calc}
\usepackage[all]{xy}
\usepackage{graphicx}
\usepackage{tikz}

  \newcommand{\R}{\mathbb{R}}
\newcommand{\C}{\mathbb{C}}

\newtheorem{thm}{Theorem}
\newtheorem{rmk}{Remark}

\newtheorem{proposition}{Proposition}
\newtheorem{lem}{Lemma}

\title{Cubic Perturbations of Symmetric elliptic Hamiltonians of degree four in a Complex domain}
\author{
Bassem Ben Hamed\\
Ecole Nationale d'Electronique et des T\'el\'ecommunications de Sfax \\Route de Tunis km 10, BP 1163, 3021 Sfax, Tunisie.\\
E-mail: \texttt{bassem.benhamed@gmail.com}\\
\\
Ameni Gargouri\\
Facult\'e des Sciences de Sfax, D\'epartement de Math\'ematiques\\
BP 1171, 3000 Sfax, Tunisie.\\
E-mail: \texttt{ameni.gargouri@gmail.com}\\
\\
Lubomir Gavrilov\\
Institut de Math\'{e}matiques de Toulouse, UMR 5219\\
Universit\'{e}  de Toulouse,  31062 Toulouse,  France.\\
E-mail: \texttt{lubomir.gavrilov@math.univ-toulouse.fr}}

\date{}
\begin{document}
\maketitle
\date
\begin{abstract}
We consider arbitrary one-parameter cubic deformations of the Duffing oscillator $x''=x-x^3$. In the case when the first Melnikov function $M_1$ vanishes, but $M_2\neq 0$  we compute the general form of $M_2$ and   study its zeros in a suitable complex domain.

\end{abstract}

\section{Introduction}
\label{section0}

\noindent

Consider the   perturbed Duffing oscillator
\vskip0.2cm
\begin{eqnarray}\label{f}
X_{\epsilon}:
\left\{\begin{array}{ccl} \dot{x}&=&H_{y} +\epsilon f(x,y,\epsilon) \\
\dot{y}&=&-H_{x}+\epsilon g(x,y,\epsilon)
\end{array}\right.
\end{eqnarray}
Where $f(x,y,\epsilon)$, $g(x,y,\epsilon)$ are arbitrary cubic polynomials:
\begin{eqnarray*}
f(x,y,\epsilon)=\lambda_{0}+\lambda_{1}x+\lambda_{2}y+\lambda_{3}xy+\lambda_{4}x^{2}+\lambda_{5}y^{2}+\lambda_{6}x^{2}y+\lambda_{7}xy^{2}+\lambda_{8}x^{3}+\lambda_{9}y^{3}\\
g(x,y,\epsilon)=\gamma_{0}+\gamma_{1}x+\gamma_{2}y+\gamma_{3}xy+\gamma_{4}x^{2}+\gamma_{5}y^{2}+\gamma_{6}x^{2}y+\gamma_{7}xy^{2}+\gamma_{8}x^{3}+\gamma_{9}y^{3} \end{eqnarray*}
\noindent  and the parameters $\lambda_{i}=\sum_{j\geq0}\lambda_{i,j}\epsilon^{j}$, $\lambda_{i,j}\in\R$ and $\gamma_{i}=\sum_{j\geq0}\gamma_{i,j}\epsilon^{j}$, $\gamma_{i,j}\in\R$ are analytic functions of the small parameter $\epsilon$.
\noindent For $\epsilon= 0$   the system
is integrable, with a first integral
$$
H= \frac{y^2}{2} - \frac{x^2}{2} + \frac{x^4}{4}
$$
 and its phase portrait is shown on fig.\ref{fig1}.  The exterior period annulus and the two interior period annuli on fig.\ref{fig1}  give rise to three displacement maps of $X_\varepsilon$ with power series expansions of the form
 $$
 d(h,\epsilon)=\epsilon^k M_{k}(h)+\epsilon^{k+1} M_{k+1}(h)+....
 $$
 (where as usual  $h$ is the restriction of $H$ on a suitable cross-section to the period annulus). The number of the limit cycles bifurcating from each period annulus is bounded by the number of the zeros of the first non-vanishing Melnikov function $M_k$.  According to the Poincar\'e-Pontryagin formula
 $$ M_{1}(h)=\int_{H=h}\omega_{0}dx=\int_{H=h}g(x,y,0)dx-f(x,y,0)dy$$
 is a complete elliptic integral. Its zeros correspond to limit cycles bifurcating from the corresponding period annulus.  It is well known, that in our case the first non-vanishing Melnikov function $M_k$ is a complete elliptic integral, see \cite[Corollary 1]{Gav05}, and \cite{gavr99,gavr05}, and its  general form has been established  in formula (23) and Theorem 3 of \cite{Gav05}, as a linear combination of complete elliptic integrals.

 Our first result is an explicit formula for the second Melnikov function $M_2$, under the hypothesis that $M_1$ is identically zero, see Proposition \ref{m2int} and Proposition \ref{m2ext}. The main tool is the Iliev formula for $M_2$ \cite{Iliev98}. This formula was proved using the method of \cite{Francoise}. Our second result is an estimate for the number of the zeros of $M_2$, Lemma \ref{g1}, \ref{g3},  \ref{z1}, \ref{z3}. From this we deduce the maximal cyclicity of the period annuli, when at least $M_1$ or $M_2$ does nor vanish identically.

\begin{figure}
\begin{center}
 \def\svgwidth{9cm}
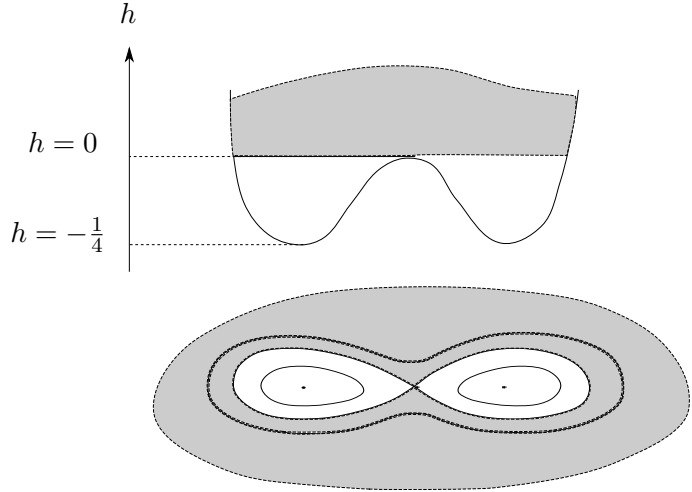
\end{center}
\caption{Phase portrait of $X_0$ and the graph of  $- \frac{x^2}{2} + \frac{x^4}{4}$}
\label{fig1}
\end{figure}

The paper is organized as follows. In section \ref{sectioncm} we compute the Melnikov functions $M_{1}(h)$ and $M_{2}(h)$ (when $M_1(h)\equiv 0$) in the interior and exterior eight-loop case see respectively. In section \ref{section1} we recall some known  Picard-Fuchs equations, which will be used later. Finally, in section  \ref{zeros} we describe the monodromy of the Abelian integrals, based on the classical Picard-Lefschetz theory, and then apply the so called Petrov trick, to obtain estimates
to the number of their  zeros in a suitable complex domain.

\section{Computation of Melnikov Functions}
\label{sectioncm}
\noindent
\vskip0.2cm
\noindent
Let $\{\gamma(h)\}_h$ be the continuous family of ovals of the non-perturbed system, where  $$\gamma(h)~\subset\{H=h\}$$ \noindent with  $h\in\Sigma=(h_{c}, h_{s})$ in the interior eight-loop case and $h\in\Sigma=(h_{s}, +\infty)$ in the exterior eight-loop case,
where $h_s=0$, $h_c=-1/4$ are the critical values of $H$.\\ \noindent
Consider the complete elliptic integrals
\begin{equation}\label{3}
\;\text{$I_{i}(h)=$}\;
\left\{
  \begin{aligned}
I_{\widetilde{\omega}_{i}}&=&\oint_{\gamma(h)} x^{i}y dx \;\text{ if }\; \Sigma=(h_{c}, h_{s})\\
I_{\widetilde{\widetilde{\omega_{i}}}}&=&\oint_{\gamma(h)} x^{i}y dx \;\text{ if }\; \Sigma=(h_{s}, +\infty)
      \end{aligned}
    \right.
\end{equation}

The Abelian integrals $I_k$, $k\geq 0$, can be expressed as linear combinations of $I_0, I_1,I_2$, with coefficients in the field $\R(h)$. \noindent In the exterior eight-loop case the symmetry $(x,y)\rightarrow (\pm x,y)$ transforms the oval $\gamma(h)$ to  $-\gamma(h)$ which implies that $I_{k}(h)\equiv0$ for odd $k$.

As well known, if we parameterize the displacement map by the Hamiltonian level h, then the following power series  expansion holds \\
\begin{eqnarray}\label{777}
d(h,\epsilon)=P(h,\epsilon)-h=\epsilon M_{1}(h)+\epsilon^{2} M_{2}(h)+...., h\in\Sigma
\end{eqnarray}
\noindent Where $P(h,\epsilon)$ is the first return map, $\Sigma$ is an open interval depending on the case under consideration.
Our first goal will be to calculate explicitly the first Melnikov function $M_{1}$ and then $M_{2}$ in (\ref{777}). We use the Iliev formula \cite{Iliev98} .\\

\noindent We denote:
\begin{eqnarray}\label{11}
f(x,y,0)=\lambda_{0,1}+\lambda_{1,1}x+\lambda_{2,1}y+\lambda_{3,1}xy+\lambda_{4,1}x^{2}+\lambda_{5,1}y^{2}+\lambda_{6,1}x^{2}y+\lambda_{7,1}xy^{2}+\lambda_{8,1}x^{3}+\lambda_{9,1}y^{3}\\
g(x,y,0)=\gamma_{0,1}+\gamma_{1,1}x+\gamma_{2,1}y+\gamma_{3,1}xy+\gamma_{4,1}x^{2}+\gamma_{5,1}y^{2}+\gamma_{6,1}x^{2}y+\gamma_{7,1}xy^{2}+\gamma_{8,1}x^{3}+\gamma_{9,1}y^{3} \end{eqnarray}

\begin{eqnarray}\label{22}
f_{\epsilon}(x,y,0)=\lambda_{0,2}+\lambda_{1,2}x+\lambda_{2,2}y+\lambda_{3,2}xy+\lambda_{4,2}x^{2}+\lambda_{5,2}y^{2}+\lambda_{6,2}x^{2}y+\lambda_{7,2}xy^{2}+\lambda_{8,2}x^{3}+\lambda_{9,2}y^{3}\\
g_{\epsilon}(x,y,0)=\gamma_{0,2}+\gamma_{1,2}x+\gamma_{2,2}y+\gamma_{3,2}xy+\gamma_{4,2}x^{2}+\gamma_{5,2}y^{2}+\gamma_{6,2}x^{2}y+\gamma_{7,2}xy^{2}+\gamma_{8,2}x^{3}+\gamma_{9,2}y^{3}
\end{eqnarray}

We recall, that non-perturbed Hamiltonian system has two bounded (interior) period annuli and one unbounded (exterior) period annulus.
\subsection{Computation of $M_1$}
\subsubsection{The interior Duffing oscillator}
\begin{proposition}
The first Melnikov functions $M_{1}$ for the perturbed interior Duffing oscillator have the form
\begin{eqnarray}\label{first}
M_{1}(h)&= &\alpha_{0}(h)I_{0}+\alpha_{1} I_{1}+\alpha_{2} I_{2} \end{eqnarray}
where
$$ \alpha_{0}(h)=c_{0}+c_{1}h,   \alpha_{1}=2\lambda_{4,1}+\gamma_{3,1} ,    \alpha_{2}=c_{2},$$
and
$$c_{0}=\lambda_{1,1}+\gamma_{2,1},  c_{1}=\frac{4}{7}(\lambda_{7,1}+3\gamma_{9,1}),  c_{2}=\gamma_{6,1}+3\lambda_{8,1}+\frac{1}{7}\lambda_{7,1}+\frac{3}{7}\gamma_{9,1}.$$
\end{proposition}

\begin{proof}
\noindent It is well known that :\\
$$M_{1}(h)=\int_{H=h}\omega_{0}dx=\int_{H=h}g(x,y,0)dx-f(x,y,0)dy$$.\\
\noindent where
\begin{eqnarray*}
 \int_{H=h}g(x,y,0)dx=\int_{H=h}[y(\gamma_{2,1}+\gamma_{3,1}x+\gamma_{6,1}x^{2})+y^{2}(\gamma_{5,1}+\gamma_{7,1}x)+\gamma_{9,1}y^{3}]dx
 \end{eqnarray*}
 \begin{eqnarray*}
-\int_{H=h}f(x,y,0)dy=-\int_{H=h}[\lambda_{1,1}x+\lambda_{3,1}xy+\lambda_{4,1}x^{2}+\lambda_{6,1}x^{2}y+
\lambda_{7,1}xy^{2}+\lambda_{8,1}x^{3}]dy
\end{eqnarray*}
\noindent Or\\ \\
\noindent $xdy=d(xy)-ydx$, \ \ $xydy=d(x\frac{y^{2} }{2})-\frac{y^{2}}{2}dx$, \ \ $x^{2}dy=d(x^{2}y)-2xydx$ \\ \\
\noindent $x^{2}ydy=d(x^{2}\frac{y^{2}}{2})-xy^{2}dx$, \ \ $xy^{2}dy=d(x\frac{y^{3}}{3})-\frac{y^{3}}{3}dx$, \ \ $x^{3}dy=d(x^{3}y)-3x^{2}ydx$\\

\noindent Therefore we can rewrite $\int_{H=h}\omega$ in the form $\int_{H=h}\omega=dQ(x,y,0)+yq(x,y,0)dx$ with $$Q(x,y,0)=\gamma_{0,1}x+\frac{\gamma_{1,1}}{2}x^{2}+\frac{\gamma_{4,1}}{3}x^{3}+\frac{\gamma_{8,1}}{4}x^{4}$$ and \\ $$yq(x,y,0)=[(\lambda_{1,1}+\gamma_{2,1})+(\gamma_{3,1}+2\lambda_{4,1})x+(\gamma_{6,1}+3\lambda_{8,1})x^{2}]y$$\noindent $$+[(\gamma_{5,1}+\frac{\lambda_{3,1}}{2})+(\lambda_{6,1}+\gamma_{7,1})x]y^{2}+(\gamma_{9,1}+\frac{\lambda_{7,1}}{3})y^{3}.$$\\ \noindent Then\\
\begin{eqnarray*}
M_{1}(h)=(\gamma_{2,1}+\lambda_{1,1})I_{0}+(\gamma_{3,1}+2\lambda_{4,1})I_{1}+(\gamma_{6,1}+3\lambda_{8,1})I_{2}+(\frac{\lambda_{7,1}}{3}+\gamma_{9,1})\int_{H=h}y^{3}dx.
\end{eqnarray*}
\noindent and $$\int_{H=h}y^{3}dx=\int_{H=h}y(2h+x^{2}-\frac{x^{4}}{2})=2hI_{0}+I_{2}-\frac{I_{4}}{2}=\frac{12h}{7}I_{0}+ \frac{3}{7}I_{2}$$
\noindent   implies (\ref{first})\end{proof}
\vskip0.2cm
\subsubsection{The exterior Duffing oscillator}

\begin{proposition}
The first Melnikov functions $ M_{1}$ for the perturbed exterior Duffing oscillator have the form
\begin{eqnarray}\label{firstext}
M_{1}(h)&= &\alpha_{0}(h)I_{0}+\alpha_{2} I_{2} \end{eqnarray}
 where
 $$ \alpha_{0}(h)=c_{0}+c_{1}h, \alpha_{2}=c_{2}$$
$$c_{0}=\lambda_{1,1}+\gamma_{2,1}, c_{1}=\frac{4}{7}(\lambda_{7,1}+3\gamma_{9,1}), c_{2}=\gamma_{6,1}+3\lambda_{8,1}+\frac{1}{7}\lambda_{7,1}+\frac{3}{7}\gamma_{9,1} .$$
\end{proposition}
\begin{proof}
\noindent It is similar to the   proof   in the exterior case, with the only exception that  $I_{1}=0$.
\end{proof}

\subsection{Computation of $M_2$}
\label{aaa}
\noindent
If $M_1=0$,
the   Iliev formula \cite{Iliev98} for the second Melnikov function  $M_{2}(h)$ reads
\begin{eqnarray*}
M_{2}(h)&=&\int_{H=h}[G_{1h}(x, y)P_{2}(x, h)-G_{1}(x, y)P_{2h}(x, h)]dx \\
& & -\int_{H=h}\frac{F(x,y)}{y}(f_{x}(x, y, 0)+g_{y}(x, y, 0))dx\\
& &+\int_{H=h}g_{\epsilon}(x, y, 0)dx-f_{\epsilon}(x, y, 0)dy
\end{eqnarray*}
where
$$ F(x,y)= \int^{y}_{0}f(x,s,0)ds-\int^{x}_{0}g(s,0,0)ds, \ \ \ G(x, y)= g(x, y, 0)+ F_{x}(x,y) $$
and
 $G_{1}(x, y)$, $G_{2}(x, y)$ are the odd and even parts of $G(x, y)$ with respect to y. Thus if
$$G(x, y)=y[(\lambda_{1,1}+\gamma_{2,1})+ (\gamma_{3,1}+2\lambda_{4,1})x+(\gamma_{6,1}+3\lambda_{8,1})x^{2}+y^{2}(\gamma_{9,1}+\frac{\lambda_{7,1}}{3})]+y^{2}[(\gamma_{5,1}+
\frac{\lambda_{3,1}}{2})+(\gamma_{7,1}+\lambda_{6,1})x]$$
then\\ \noindent $G(x, y)=G_{1}(x, y)+ G_{2}(x, y)$, \ \ \ $G_{1}(x, y)=yp_{1}(x,y^{2})$, \ \ \  $G_{2}(x, y)=p_{2}(x,y^{2})$\\
\noindent with \\ \noindent $p_{1}(x,y^{2})=(\lambda_{1,1}+\gamma_{2,1})+ (\gamma_{3,1}+2\lambda_{4,1})x+(\gamma_{6,1}+3\lambda_{8,1})x^{2}+y^{2}(\gamma_{9,1}+\frac{\lambda_{7,1}}{3})$\\ \noindent and \\ \noindent $p_{2}(x,y^{2})=y^{2}[(\gamma_{5,1}+
\frac{\lambda_{3,1}}{2})+(\gamma_{7,1}+\lambda_{6,1})x]$\\ \\
\noindent $\bullet$ $P_{2}(x, h)$ is the polynomial $P_{2}(x, h)=\int^{x}_{0}P_{2}(s, 2h+2U(s))ds= 2hx(\gamma_{5,1}+\frac{\lambda_{3,1}}{2})+hx^{2}(\gamma_{7,1}+\lambda_{6,1}) \noindent\\ \\+\frac{x^{3}}{3}(\gamma_{5,1}+\frac{\lambda_{3,1}}{2})
+\frac{x^{4}}{4}(\gamma_{7,1}+\lambda_{6,1})-\frac{x^{5}}{10}(\gamma_{5,1}+\frac{\lambda_{3,1}}{2})-\frac{x^{6}}{12}(\gamma_{7,1}+\lambda_{6,1})$\\
\noindent\\ $\bullet$ We note that
\begin{equation}\label{tu}
 G_{1h}(x, y)=G_{1y}(x, y)/y= (\lambda_{1,1}+\gamma_{2,1})/y +(\gamma_{3,1}+2\lambda_{4,1})\frac{x}{y}+(\gamma_{6,1}+3\lambda_{8,1})\frac{x^{2}}{y}+3y(\gamma_{9,1}+\frac{\lambda_{7,1}}{3}).
\end{equation}
\noindent $\bullet$
$g(x,y,0)=\gamma_{0,1}+\gamma_{1,1}x+\gamma_{4,1}x^{2}+\gamma_{8,1}x^{3}+y(\gamma_{2,1}+\gamma_{3,1}x+\gamma_{6,1}x^{2})+y^{2}(\gamma_{5,1}+\gamma_{7,1}x)+
\gamma_{9,1}y^{3}$\\ \\
\noindent $\bullet$ $F(x,y)=\int^{y}_{0}f(x,s,0)ds-\int^{x}_{0}g(s,0,0)ds=\lambda_{0,1}y+\lambda_{1,1}xy-\gamma_{0,1}x+\frac{\lambda_{2,1}}{2}y^{2}-
\frac{\gamma_{1,1}}{2}x^{2}+\frac{\lambda_{3,1}}{2}xy^{2}$ \noindent\\ \\$+\lambda_{4,1}x^{2}y-\frac{\gamma_{4,1}}{3}x^{3}+\frac{\lambda_{5,1}}{3}y^{3}+
\frac{\lambda_{6,1}}{2}x^{2}y^{2}+\frac{\lambda_{7,1}}{3}xy^{3}+\lambda_{8,1}x^{3}y+\frac{\lambda_{9,1}}{4}y^{4}-\frac{\gamma_{8,1}}{4}x^{4}$.\\
\noindent\\ $\bullet$ Then\\ \\ \noindent$-\frac{F(x,y)}{y}=-\frac{1}{y}(\int^{y}_{0}f(x,s,0)ds-\int^{x}_{0}g(s,0,0)ds)=-\lambda_{0,1}-\lambda_{1,1}x+\gamma_{0,1}\frac{x}{y}-\frac{\lambda_{2,1}}{2}y+
\frac{\gamma_{1,1}}{2}\frac{x^{2}}{y}-\frac{\lambda_{3,1}}{2}xy$ \noindent\\ \\ $-\lambda_{4,1}x^{2}+\frac{\gamma_{4,1}}{3}\frac{x^{3}}{y}-\frac{\lambda_{5,1}}{3}y^{2}-
\frac{\lambda_{6,1}}{2}x^{2}y-\frac{\lambda_{7,1}}{3}xy^{2}-\lambda_{8,1}x^{3}-\frac{\lambda_{9,1}}{4}y^{3}+\frac{\gamma_{8,1}}{4}\frac{x^{4}}{y}$.\\

\noindent In fact:
$$\int^{y}_{0}f(x,s,0)ds=\lambda_{0,1}y+\lambda_{1,1}xy+\lambda_{2,1}\frac{y^{2}}{2}+\lambda_{3,1}x\frac{y^{2}}{2}+\lambda_{4,1}x^{2}y+\lambda_{5,1}\frac{y^{3}}{3}+\lambda_{6,1}x^{2}\frac{y^{2}}{2}+\lambda_{7,1}\frac{xy^{3}}{3}+
\lambda_{8,1}x^{3}y+\lambda_{9,1}\frac{y^{4}}{4}$$

$$\int^{x}_{0}g(s,0,0)ds=\gamma_{0,1}x+\gamma_{1,1}\frac{x^{2}}{2}+\gamma_{4,1}\frac{x^{3}}{3}+\gamma_{8,1}\frac{x^{4}}{4}$$

\vskip0.2cm
\subsubsection{The interior Duffing oscillator}
Lemma \ref{14} implies easely the linear independence
of the functions $I_{0}(h)$, $ hI_{0}(h)$,  $I_{1}(h)$ and $I_{2}(h)$. As $I_{1}=c(4h-3)$
then $M_1=0$
implies
\begin{eqnarray}\label{peru}
\lambda_{1,1}+\gamma_{2,1}=0\\
\lambda_{7,1}+3\gamma_{9,1}=0\\
2\lambda_{4,1}+\gamma_{3,1}=0\\
\gamma_{6,1}+3\lambda_{8,1}=0
\end{eqnarray}

\begin{proposition}
\label{m2int}
The function $M_{2}(h)$ has the follows form:
\begin{eqnarray}\label{mm}
M_{2}(h)=(\alpha_{0}+4\alpha_{1}h)I_{0}+(\beta_{0}+4h\beta_{1})I_{1}+\rho I_{2}
\end{eqnarray}
\noindent where $$\alpha_{0}=-\lambda_{0,1}(\lambda_{3,1}+2\gamma_{5,1})+\lambda_{1,2}+\gamma_{2,2}$$
\noindent $$\alpha_{1}=(\lambda_{3,1}+2\gamma_{5,1})(-\frac{1}{7}\lambda_{8,1}-\lambda_{5,1})$$\\
\noindent $\beta_{0}= -(\lambda_{3,1}+2\gamma_{5,1})(\lambda_{1,1}-\frac{1}{8}\lambda_{7,1})+2(\lambda_{6,1}+\gamma_{7,1})(\lambda_{0,1}+2\lambda_{4,1}-2\lambda_{7,1})+
2\lambda_{4,2}+\gamma_{3,2}$
\noindent \\ \\$\beta_{1}=-\frac{1}{2}\lambda_{7,1}(\lambda_{3,1}+2\gamma_{5,1})+3\lambda_{7,1}(\lambda_{6,1}+\gamma_{7,1})$
\noindent \\ \\$\rho= (\lambda_{3,1}+2\gamma_{5,1})(\lambda_{4,1}-\frac{1}{7}\lambda_{5,1}-\frac{8}{7}\lambda_{8,1})-2\lambda_{1,1}(\lambda_{6,1}+\gamma_{7,1})+\gamma_{6,2}+3\lambda_{8,2}+\frac{1}{7}\lambda_{7,2}+\frac{3}{7}\gamma_{9,2}$.
\end{proposition}
\begin{proof}
\noindent According to the   Iliev formula
\begin{eqnarray*}
M_{2}& = &\int_{H=h}[G_{1h}(x, y)P_{2}(x, h)-G_{1}(x, y)P_{2h}(x, h)]dx\\
& & -\int_{H=h}\frac{F(x,y)}{y}(f_{x}(x, y, 0)+g_{y}(x, y, 0))dx\\ & &+\int_{H=h}g_{\epsilon}(x, y, 0)dx
 -f_{\epsilon}(x, y, 0)dy
\end{eqnarray*}
\noindent where \\
$$\int_{H=h}g_{\epsilon}dx-f_{\epsilon}dy=[\lambda_{1,2}+\gamma_{2,2}+ \frac{4}{7}(\lambda_{7,2}+3\gamma_{9,2})h]I_{0}+(2\lambda_{4,2}+\gamma_{3,2})I_{1} +[\gamma_{6,2}+3\lambda_{8,2}+\frac{1}{7}\lambda_{7,2}+\frac{3}{7}\gamma_{9,2}]I_{2}$$
\noindent By using (\ref{peru}), (13), (14) and (15)  we have: $p_{1}(x, y^{2})=0$ then $G_{1}(x,y)=0$ and (\ref{tu}) becomes zero.\\

\noindent Thus\\ \\
$M_{2}(h)=-\int_{H=h}\frac{F(x,y)}{y}(f_{x}+g_{y})dx+\oint_{H=h}g_{\epsilon}dx-f_{\epsilon}dy$\noindent\\$=
+[-\lambda_{0,1}(\lambda_{3,1}+2\gamma_{5,1})+\lambda_{1,2}+\gamma_{2,2}]I_{0}+[(\lambda_{3,1}+2\gamma_{5,1})(-\frac{1}{7}\lambda_{8,1}-\lambda_{5,1})]4hI_{0} $
\\ \\\noindent $+[-(\lambda_{3,1}+2\gamma_{5,1})(\lambda_{1,1}-\frac{1}{8}\lambda_{7,1})+2(\lambda_{6,1}+\gamma_{7,1})(\lambda_{0,1}+2\lambda_{4,1}-2\lambda_{7,1})+
2\lambda_{4,2}+\gamma_{3,2}]I_{1}$
\\ \\ \noindent $+[-\frac{1}{2}\lambda_{7,1}(\lambda_{3,1}+2\gamma_{5,1})+3\lambda_{7,1}(\lambda_{6,1}+\gamma_{7,1})]hI_{1}$
\\ \\  \noindent $+[(\lambda_{3,1}+2\gamma_{5,1})(\lambda_{4,1}-\frac{1}{7}\lambda_{5,1}-\frac{8}{7}\lambda_{8,1})-2\lambda_{1,1}(\lambda_{6,1}+\gamma_{7,1})+\gamma_{6,2}+3\lambda_{8,2}+\frac{1}{7}\lambda_{7,2}+\frac{3}{7}\gamma_{9,2}]I_{2}$.
\end{proof}
\subsubsection{The exterior Duffing oscillator}
\vskip0.2cm
\noindent
In a way similar to the interior Duffing oscillator, we conclude that if $M_1=0$ then
\begin{eqnarray}\label{per}
\lambda_{1,1}+\gamma_{2,1}=0\\
\lambda_{7,1}+3\gamma_{9,1}=0\\
\gamma_{6,1}+3\lambda_{8,1}=0
\end{eqnarray}

\begin{proposition}
\label{m2ext}
The function $M_{2}(h)$ has the follows form:
\begin{eqnarray}\label{second}
M_{2}(h)=(4h+1)^{-1}[(\alpha_{0}+4\alpha_{1}h+\alpha_{2}h^{2})I_{0}+(\beta_{0}+4h\beta_{1})I_{2}]
\end{eqnarray}
\noindent where $$\alpha_{0}=-\lambda_{0,1}(\lambda_{3,1}+2\gamma_{5,1})+\lambda_{1,2}+\gamma_{2,2}-\gamma_{0,1}(2\lambda_{4,1}+\gamma_{3,1})$$
\noindent $$\alpha_{1}=-\lambda_{0,1}(\lambda_{3,1}+2\gamma_{5,1})+\lambda_{1,2}+\gamma_{2,2}-
\frac{4}{7}(\lambda_{5,1}(\lambda_{3,1}+2\gamma_{5,1}))+\frac{4}{7}(\lambda_{7,2}+3\gamma_{9,2})-\frac{8}{7}(\lambda_{8,1}(\lambda_{6,1}+\gamma_{7,1}))+
\frac{\gamma_{4,1}}{3}(2\lambda_{4,1}+\gamma_{3,1})$$ \noindent $$+\frac{8}{15}(\gamma_{3,1}+2\lambda_{4,1})(\gamma_{5,1}+\frac{\lambda_{3,1}}{2})$$
\noindent $$\alpha_{2}=-
\frac{4}{7}(\lambda_{5,1}(\lambda_{3,1}+2\gamma_{5,1}))+\frac{4}{7}(\lambda_{7,2}+3\gamma_{9,2})-\frac{8}{7}(\lambda_{8,1}(\lambda_{6,1}+\gamma_{7,1}))$$
\noindent \\ $$\beta_{0}=-[2\lambda_{4,1}\lambda_{3,1}+\frac{\lambda_{3,1}\gamma_{3,1}}{2}+2\lambda_{4,1}\gamma_{5,1}+2\lambda_{1,1}\lambda_{6,1}+2\lambda_{1,1}\gamma_{7,1}
-\gamma_{6,2}-3\lambda_{8,2}-\frac{1}{7}\lambda_{7,2}-\frac{3}{7}\gamma_{9,2}$$ \noindent $$+ \frac{\lambda_{5,1}}{7}(\lambda_{3,1}+2\gamma_{5,1})+\frac{16}{7}(\lambda_{8,1}(\lambda_{6,1}+
\gamma_{7,1}))-5(2\lambda_{4,1}+\gamma_{3,1})(\frac{\gamma_{4,1}}{3}+\gamma_{0,1})+\frac{17}{15}(\gamma_{3,1}+2\lambda_{4,1})(\gamma_{5,1}+
\frac{\lambda_{3,1}}{2})]$$\\

\noindent $$\beta_{1}=-[2\lambda_{4,1}\lambda_{3,1}+\frac{\lambda_{3,1}\gamma_{3,1}}{2}+2\lambda_{4,1}\gamma_{5,1}+2\lambda_{1,1}\lambda_{6,1}+2\lambda_{1,1}\gamma_{7,1}
-\gamma_{6,2}-3\lambda_{8,2}-\frac{1}{7}\lambda_{7,2}-\frac{3}{7}\gamma_{9,2}$$  \noindent $$+\frac{\lambda_{5,1}}{7}(\lambda_{3,1}+2\gamma_{5,1})+
\frac{16}{7}(\lambda_{8,1}(\lambda_{6,1}+\gamma_{7,1}))-\frac{1}{5}(\gamma_{3,1}+2\lambda_{4,1})(\gamma_{5,1}+\frac{\lambda_{3,1}}{2})]$$
\end{proposition}

\begin{proof}
\noindent The same way of proof of property 3, We use also the formula of  Iliev \cite{Iliev98}:
$$M_{2}(h)=\int_{H=h}[G_{1h}(x, y)P_{2}(x, h)-G_{1}(x, y)P_{2h}(x, h)]dx$$  \noindent$$ -\int_{H=h}\frac{F(x,y)}{y}(f_{x}(x, y, 0)+g_{y}(x, y, 0))dx$$ \noindent$$+\int_{H=h}g_{\epsilon}(x, y, 0)dx-f_{\epsilon}(x, y, 0)dy
$$
\noindent \\ By using (\ref{per}), (18) and (19) we have $p_{1}(x, y^{2})=(\gamma_{3,1}+2\lambda_{4,1})x$  and (\ref{tu}) becomes $$ G_{1h}(x, y)=(\gamma_{3,1}+2\lambda_{4,1})\frac{x}{y}$$\\
\noindent Then\\ \\ $\bullet$ $\int_{H=h}[G_{1h}(x, y)P_{2}(x, h)-G_{1}(x, y)P_{2h}(x, h)]dx=-2(\gamma_{3,1}+2\lambda_{4,1})(\gamma_{5, 1}+\frac{\lambda_{3,1}}{2})I_{2}$\noindent$$+2h(\gamma_{3,1}+2\lambda_{4,1})(\gamma_{5, 1}+\frac{\lambda_{3,1}}{2})I'_{2}+\frac{1}{3}(\gamma_{3,1}+2\lambda_{4,1})(\gamma_{5, 1}+\frac{\lambda_{3,1}}{2})I'_{4}-\frac{1}{10}(\gamma_{3,1}+2\lambda_{4,1})(\gamma_{5, 1}+\frac{\lambda_{3,1}}{2})I'_{6}$$
\noindent and by using the Picards-Fuchs equations (see for instance \cite{Deligne}, for more details) we have \\ \\  \noindent $$I'_{2}=(4h+1)^{-1}(5I_{2}-I_{0})$$ \noindent $$I'_{4}=(4h+1)^{-1}(4hI_{0}+5I_{2})$$ \noindent $$I'_{6}=(4h+1)^{-1}[\frac{4}{3}(4h+1)I_{2}+\frac{4}{3}h(5I_{2}-I_{0})+\frac{4}{3}(4hI_{0}+5I_{2})]$$.\\
\noindent\\  Then $$\int_{H=h}[G_{1h}(x, y)P_{2}(x, h)-G_{1}(x, y)P_{2h}(x, h)]dx $$ \noindent $$=(4h+1)^{-1}(\gamma_{3,1}+2\lambda_{4,1})(\gamma_{5, 1}+\frac{\lambda_{3,1}}{2})[(\frac{4}{5}h-\frac{17}{15})I_{2}-\frac{32}{15}hI_{0}]$$\\

\vskip0.2cm
\noindent
$$\int_{H=h}g_{\epsilon}dx-f_{\epsilon}dy=[\lambda_{1,2}+\gamma_{2,2}+ \frac{4}{7}(\lambda_{7,2}+3\gamma_{9,2})h]I_{0} +[\gamma_{6,2}+3\lambda_{8,2}+\frac{1}{7}\lambda_{7,2}+\frac{3}{7}\gamma_{9,2}]I_{2}$$
\noindent  By using (\ref{per}) we have also $$(f_{x}+g_{y})=(2\lambda_{4,1}+\gamma_{3,1})x+(\lambda_{3,1}+2\gamma_{5,1})y+2(\lambda_{6,1}+\gamma_{7,1})xy$$
\noindent\\ Then\\ \\ $\bullet$  $-\int_{H=h}\frac{F(x,y)}{y}(f_{x}+g_{y})dx+\oint_{H=h}g_{\epsilon}dx-f_{\epsilon}dy$\noindent\\ \\$=
[-\lambda_{0,1}(\lambda_{3,1}+2\gamma_{5,1})+\lambda_{1,2}+\gamma_{2,2}]I_{0}$\\ \\ \noindent $+[-\frac{4}{7}(\lambda_{5,1}(\lambda_{3,1}+2\gamma_{5,1}))+\frac{4}{7}(\lambda_{7,2}+3\gamma_{9,2})-\frac{8}{7}(\lambda_{8,1}(\lambda_{6,1}+\gamma_{7,1}))]hI_{0}$
\\ \\   \noindent  $-[2\lambda_{4,1}\lambda_{3,1}+\frac{\lambda_{3,1}\gamma_{3,1}}{2}+2\lambda_{4,1}\gamma_{5,1}+2\lambda_{1,1}\lambda_{6,1}+2\lambda_{1,1}\gamma_{7,1}
-\gamma_{6,2}-3\lambda_{8,2}-\frac{1}{7}\lambda_{7,2}-\frac{3}{7}\gamma_{9,2}+\frac{\lambda_{5,1}}{7}(\lambda_{3,1}+2\gamma_{5,1})$\noindent \\ \\$+\frac{16}{7}(\lambda_{8,1}(\lambda_{6,1}+\gamma_{7,1}))]I_{2}$\\ \\ \noindent $+\gamma_{0,1}(2\lambda_{4,1}+\gamma_{3,1})I'_{2}+
\frac{\gamma_{4,1}}{3}(2\lambda_{4,1}+\gamma_{3,1})I'_{4}$ \\

\noindent Or we have $I'_{2}=(4h+1)^{-1}(5I_{2}-I_{0})$ and $I'_{4}=(4h+1)^{-1}(4hI_{0}+5I_{2})$.\\
\noindent Then we can obtain by using the above information proposition 2.
\end{proof}


\section{Picards-Fuchs equations }
\label{section1}
The results of this section are known, or can be easily deduced, see \cite{Iliev99, jezo94, Zol91}.

First we note that the affine complex algebraic curve
$$
\Gamma_h= \{(x,y)\in \C: H(x,y)=h \}
$$
is smooth for $h \neq 0,-1/4$ and has the topological type of a torus with two removed points $\infty^\pm$ (at "infinity"). Its homology group is therefore of rang three, the corresponding De Rham group has for generators the (restrictions of) polynomial differential one-forms
$$
ydx, \;xy dx, \;x^2ydx
$$
which are also generators of the related Brieskorn-Petrov $\C[h]$-module \cite{gavr98}.

\begin{lem}
\noindent\\The integrals $I_{i}$, $i=0,2$, satisfy the following system of Picard-Fuchs:
\begin{eqnarray}\label{area}
I_0(h)&=&\frac{4}{3}hI'_0(h)+\frac{1}{3}I'_2(h) \\
I_2(h)&=&\frac{4}{15}hI'_0(h)+\left(\frac{4}{5}h+\frac{4}{15}\right)I'_2(h)
\end{eqnarray}
\end{lem}
\begin{proof}
\noindent See proof of lemma$5$ of Petrov \cite{Pet87} for details.
\end{proof}
\noindent The above equations imply the following asymptotic expansions near $h=0$ (they agree with the Picard-Lefshetz formula)

\begin{lem}\label{14}
\noindent The integrals $I_i$, $i=0,2$, have the following asymptotic expansions in the neighborhood of $h=0$:
\noindent
\begin{eqnarray*}
I_{0}(h)&=&(-h+\frac{3}{8}h^{2}-\frac{35}{64}h^{3}+...)\ln h+\frac{4}{3}+a_{1}h+a_{2}h^{2}+...\\
I_{2}(h)&=&(\frac{1}{2}h^{2}-\frac{5}{8}h^{3}-\frac{315}{256}h^{4}...)\ln h+\frac{16}{15}+4h+b_{2}h^{2}+...\\
\end{eqnarray*}
\end{lem}
\begin{proof}
\noindent For proof see \cite{Gav13}.
\end{proof}
\section{Zeros of Abelian integrals in a complex domain}
\label{zeros}
\noindent
 Our goal will be to find the upper bounds number of the zeroes of the Abelian integrals defined in \eqref{first} and \eqref{second} on the interval of existence of the ovals $\{\gamma(h)\}$.\\
\noindent All families of cycles will depend continuously on a parameter $h$ and will be defined without ambiguity in the complex half-plane ${h: Im(h)>0}$. This will allows a continuation on $\C$ along any curve avoiding the real critical values of $H$.\\
\noindent We use the well known Petrov method which is based on the argument principle. This gives an information on the complex limit cycles of the system in the interior and exterior eight-loop, see later lemmas \ref{g1}, \ref{z1} and \ref{z3}, respectively.\\
\noindent Our primary motivation was that the complex methods we use, are necessary to understand the bifurcations from the separatrix eight-loop.  Another reason is, that the complexity of the bifurcation set of $M_{1}$, $M_{2}$ in a complex domain is directly related to the number of the zeros of $M_{1}$, $M_{2}$. This observation can be possibly  generalized to higher genus curves.\\

\subsection{The interior eight-loop case}
\label{int}
\noindent In this section, we consider the interior eight-loop case, with period annulus as shown in fig.\ref{figa} (hatched part).
Let $\gamma(h)~\subset\{H=h\}$ be the continuous family of ovals of the non-perturbed system defined on the maximal open interval $\Sigma=(h_{c}, h_{s})$,
where for $h=h_{c}=-\frac{1}{4}$ the oval degenerates into two centers $\delta_{-1}$, $\delta_{1}$ respectively at the singular point $(-1,0)$, $(1, 0)$  and for $h=h_{s}=0$ every oval $\delta_{-1}$ or $\delta_{1}$ becomes a homoclinic loop of the Hamiltonian $dH=0$.\\
\noindent The family ${\delta_{h}}$ represents a continuous family of cycles vanishing at the centers $\delta_{-1}$ and $\delta_{1}$.
\begin{figure}
\begin{center}
 \def\svgwidth{9cm}
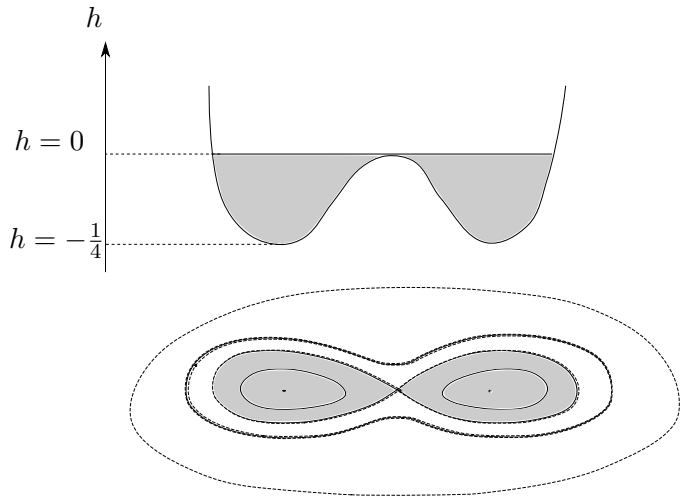
\end{center}
\caption{Phase portrait of $X_0$ and the graph of  $- \frac{x^2}{2} + \frac{x^4}{4}$}
\label{figa}
\end{figure}
\begin{thm}
\label{cyclicityint}
The maximal cyclicity of the interior
 period annulus $\{(x,y)\in \R^2: H(x,y)>0 \}$ of $dH=0$ with respect to   one-parameter analytic deformation (\ref{f}) is
\begin{description}
\item[(i) ]  three, if $M_1\neq 0$.
\item[(ii)] four, if $M_1=0$ but $M_2\neq 0$.
\end{description}
 \end{thm}

\subsubsection{The monodromy of Abelian integrals}
\label{monodromyint}
\vskip0.2cm
\begin{figure}
\begin{center}
 \def\svgwidth{7cm}
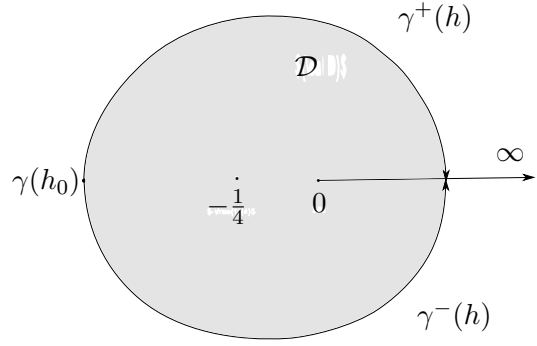
\end{center}
\caption{The analytic continuation of a cycle $\gamma(h)$ in the domain ${\cal D}= \C \setminus[0, +\infty)$}
\label{figint}
\end{figure}




\vskip0.2cm

%

\noindent An Abelian integrals $I(h)$ of the form (\ref{3}) is a multivalued analytic functions in $h\in \C$, single-valued  in the complex domain
$$
{\cal D}= \C \setminus [0,+\infty) .
$$
Moreover, along the segment $[0,+\infty)$ the integral $I(h)$ has a continuous limit when $h\in {\cal D}$ tends to a point $h_{0}\in [0,+\infty)$.  Namely, for $ h\in{\cal D}$, let $\{\gamma(h)\}_{h}$ be a continuous family of cycles, vanishing at the saddle point as $h$ tends to $h_{s}=0$.\\
The family $\{\gamma(h)\}_{h}$ has two analytic complex-conjugate continuations on $(-\infty,0)$, depending on the way in which the $h$ approaches this segment $[0,+\infty)$.
For $h\in (0,+\infty)$ denote $\gamma(h)=\gamma^+(h)$ the limit obtained when $Im (h)>0$. The cycle $\gamma^-(h)$ is defined in a similar way. It is important to note, the as $I(h)$ is real-analytic on $(-\infty,0)$, then 
$\gamma^-(h)=\overline{\gamma^+(h)}$ for $h\in (0,+\infty)$ (as follows also from the Schwarz reflection principle).
Finally,  the Picard-Lefschetz formula implies 
$$\gamma^+(h)= \gamma^-(h)+\delta_0(h) $$
where $\delta_0(h)$ is a continuous family of cycles vanishing at the saddle point as  $h\rightarrow0$.

\subsubsection{Zeros of the first return map in a complex domain}
\label{zerosint}
\begin{lem}
The first non-vanishing Poincar\'e-Pontryagin-Melnikov function (\ref{first})
has at most three zeros in the complex domain ${\cal D}$.
\label{g1}
\end{lem}
\begin{lem}
The second Poincar\'e-Pontryagin-Melnikov function (\ref{mm}) of the  first return map  has at most four zeros in the complex domain ${\cal D}$.
\label{g3}
\end{lem}

\begin{proof}[Proof of {Lemma} \ref{g1}] It follows from theorem of Petrov \cite{Pet87}. We sketch the proof:\\
\noindent We denote
$$
M_1(h) = \alpha_{0}(h)I_{0}(h)+\alpha_{1}I_{1}(h)+\alpha_{2}I_{2}(h)=\oint_{\gamma(h)}\omega=I_{\omega}(h),  h\in \cal D
$$
\noindent The monodormy of $I_{1}$ is $I_{1}$ on the ray $\{0<h\}$ (because of symmetry). Then $I_{1}(h)= a+bh=c(4h+1)$, where  $c\in\R $.
Indeed, $I_1(h)$ is univalued, of moderate growth, has no poles, vanishes at $h=-1/4$, and grows no faster that $h$ as $h$ tends to infinity. It follows that
$$M_1(h) = \alpha_{0}(h)I_{0}(h)+\alpha_{2}I_{2}(h)+c(4h+1). $$
\noindent We shall use the argument principle for analytic functions in the domain 
$${\cal D}_R= {\cal D}\cap \{ |h| \leq R \}$$
 as follows. Consider a contour encircling   ${\cal D}_R$.The number of zeros of the integral $M_1(h)$ in this domain is the number of rotations of the curve described by $M_1(h)$ about the origin as $h$ describes the contour.\\
\noindent  $\bullet$ As $h$  describes the circle $\{|h|=R\}$, for some fixed sufficiently big $R>0$, the integral $M_1(h) $ behaves as $ h^{\frac{7}{4}}$. Thus the increase of the argument of
$M_1(h)$ is close to $\frac{7\pi}{2} < 4 \pi$. \\
\noindent \\  $\bullet$ Along the   cut $[0,\R]$, the number of zeros of $M_1(h)$ about the origin is bounded by the number of zeros of the imaginary part of $M_1$, and
$$Im \,M_{1}(h)= \int_{\delta_{0}(h)}\omega, \mbox{  where  } \delta_{0}=\gamma^{+}-\gamma^{-}.$$
Therefore
$$Im \,M_{1}(h) = \alpha_{0}(h)\oint_{\delta_0(h)}y dx
 +\alpha_{2}\oint_{\delta_0(h)} x^2y dx, \;\; h\in [0,R]
 $$
 and by lemmas $7$ and $8$ of Petrov\cite{Pet87} cannot exceed 1.
  We conclude that the total increase of the argument of $M_1$ along the border of ${\cal D}_R$ can not exceed
three, which  proves  Lemma \ref{g1}.
 \end{proof}


\begin{proof}[Proof of Lemma \ref{g3}]
\noindent We denote
$$
M_{2}(h)=(\alpha_{0}+4\alpha_{1}h)I_{0}+(\beta_{0}+4h\beta_{1})I_{1}+\rho I_{2}=\oint_{\gamma(h)}w=I_{w}(h), h\in \cal D,
$$
\noindent where $\alpha_{i}$, $\beta_{i}$ and $\rho$  are defined in (\ref{mm}).\\ \noindent By making use the expression of $I_{1}=c(4h+1)$
Then\\ $$M_2(h) = \mu(h)+\alpha_{0}(h)I_{0}(h)+\alpha_{2}I_{2}(h)+\rho I_{2}$$ \\  \noindent where \\ $$\alpha_{0}(h)=\alpha_{0}+4\alpha_{1}h$$
\begin{eqnarray*}
\mu(h)&=&16ch^{2}(-\frac{1}{2}\lambda_{7,1}(\lambda_{3,1}+2\gamma_{5,1})+3\lambda_{7,1}(\lambda_{6,1}
+  \gamma_{7,1})) \\& &
+h[-12c(-\frac{1}{2}\lambda_{7,1}(\lambda_{3,1}+2\gamma_{5,1})+3\lambda_{7,1}(\lambda_{6,1}+\gamma_{7,1}))+4\beta_{0}]-3\beta_{0}
\end{eqnarray*}
and apply, as in the proof of Lemma \ref{g1}, the argument principle to $M_{2}$.
\noindent The number of zeros of the integral in this domain is the number of rotations of the curve described by $M_2(h)$ about the origin as $h$ describes the border of ${\cal D}_R$.\\ 
\noindent  $\bullet$ As $h$  describes the circle $\{|h|=R\}$; the integral $M_2(h)$ behaves as $h^{2}$ and  the increase of the argument of
$M_2(h)$ is close to $4\pi$. \\ \\
\noindent  $\bullet$ Along the   cut $(0,R]$, the number of zeros of $M_2(h)$ about the origin is bounded by the number of zeros of the imaginary part of $  M_{2}(h)$  and 
$$Im \,M_{2}(h) = \alpha_{0}(h)\oint_{\delta_0(h)}y dx
 +(\alpha_{2}+\rho) \oint_{\delta_0(h)} x^2y dx, \;\; h\in [0,R] .
 $$
Lemmas $7$ and $8$ of Petrov\cite{Pet87} imply that the number of the zeros of $Im \,M_{2}(h) $ cannot exceed $ 1$.\\
Consequently, the total number of circuits cannot exceed four, which implies Lemma \ref{g3} and hence Theorem \ref{cyclicityint}.
 \end{proof}

\vskip0.2cm

\subsection{The exterior eight-loop case}
\label{ext}

\noindent In this section we consider the exterior eight-loop case, with period annulus as shown in fig.\ref{fig1}. Let
${\gamma(h)}_{h}$ be the continuous family of exterior ovals of the non-perturbed system defined on the maximal
open interval $\Sigma=(0, +\infty)$, where $$\gamma(h)~\subset\{H=h\}$$.
\begin{thm}
\label{cyclicityext}
The maximal cyclicity of the exterior period annulus $\{(x,y)\in \R^2: H(x,y)>0 \}$ of $dH=0$ with respect to   one-parameter analytic deformation (\ref{f}) is
\begin{description}
\item[(i) ]  two, if $M_1\neq 0$.
\item[(ii)] four, if $M_1=0$ but $M_2\neq 0$.
\end{description}
\end{thm}

\begin{rmk}
\label{rem1}
The above Theorem claims that from any compact, contained in the open exterior period annulus $\{(x,y)\in \R^2: H(x,y)>0 \}$, bifurcate at most four limit cycles (if $M_2\neq 0$). It says nothing about the limit cycles bifurcating from the separatrix eight-loop or from infinity (i.e. the equator of the Poincar\'e sphere).
\end{rmk}

\subsubsection{The monodromy of Abelian integrals}
\label{monodromyext}
\begin{figure}
\begin{center}
 \def\svgwidth{9cm}
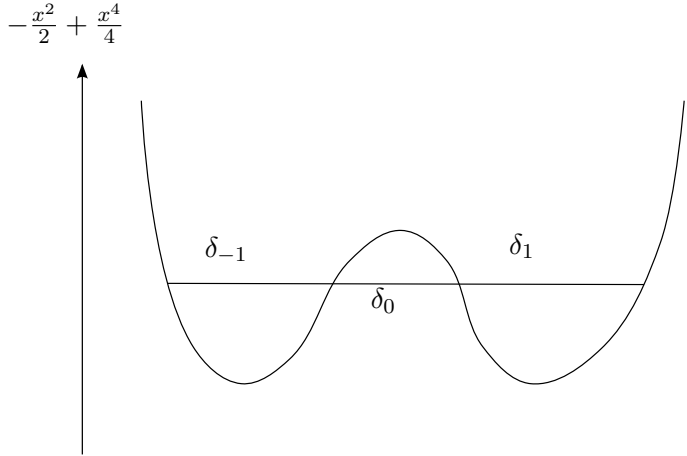
\end{center}
\caption{The vanishing cycles  $\delta_0(h), \delta_1(h), \delta_{-1}(h)$ for $-\frac{1}{4}<h<0$}
\label{figdelta}
\end{figure}
The Abelian integrals $I(h)$ of the form (\ref{3}) are multivalued functions in $h\in \C$ which become single-valued analytic functions in the complex domain
$$
{\cal D}= \C \setminus [0,-\infty) .
$$
Along the segment $[0,-\infty)$ the integrals have a continuous limit when $h\in {\cal D}$ tends to a point $h_0\in [0,-\infty)$, depending on the sign of the imaginary part of $h$. Namely, if $Im (h)>0$ we denote the corresponding limit by $I^+(h)$, and when $Im (h)>0$ by $I^-(h_0)$. We use a similar notation for the continuous limits of loops $\gamma(h)$ when $h$ tends to the segment $[0,-\infty)$. We have therefore
$$
I^\pm(h)= \int_{\gamma^\pm(h)} \omega
$$
where $\omega$ is a polynomial one-form. The monodromy $I^+(h)-I^-(h)$, $ h\in [0,-\infty)$ depends therefore on the monodromy of $\gamma(h)$ which is expressed by the Picard-Lefshetz formula. Namely, for $h\in {\cal D}$, define the continuous families of closed loops
$$\delta_0(h), \delta_1(h), \delta_{-1}(h)
$$
which vanish at the singular points $(0,0), (0,1), (0,-1)$ when $h$ tends to $0$ or $-1/4$ respectively, and in such a way that
$Im (h) > 0$, see fig.\ref{figdelta}. This defines uniquely  the homology classes of the loops, up to an orientation. From now on we suppose that the loop $\gamma(h)$ for $h>0$ is oriented by the vector field $X_0$, and that the orientation of $\delta_0(h), \delta_1(h), \delta_{-1}(h)$ are chosen in such a way that
$$
\gamma(h) = \delta_0(h) + \delta_1(h)  + \delta_{-1}(h) , \, h \in \cal D .
$$
According to the definition of the vanishing cycles
\begin{equation}
\gamma^+(h) = \delta^+_0(h) + \delta^+_1(h) + \delta^+_{-1}(h) , h\in (-\infty,0] .
\label{gplus}
\end{equation}
and the Picard-Lefschetz formula implies
\begin{equation}
\gamma^-(h)= - \delta^+_0(h) + \delta^+_1(h) + \delta^+_{-1}(h) , h\in [-1/4,0]
\label{gminus1}
\end{equation}
and
\begin{equation}
\gamma^-(h)= - \delta^+_0(h), h\in (-\infty, -1/4]
\label{gminus2}
\end{equation}
For a further use we note that
\begin{equation}
\delta^-_0(h)= \delta^+_0(h), h\in ( -1/4,+\infty)
\label{d0}
\end{equation}
\begin{equation}
\delta^-_1(h)= \delta^+_1(h),
\delta^-_{-1}(h)= \delta^+_{-1}(h),
h\in ( -\infty, 0)
\label{d1}
\end{equation}

\begin{figure}
\begin{center}
 \def\svgwidth{7cm}
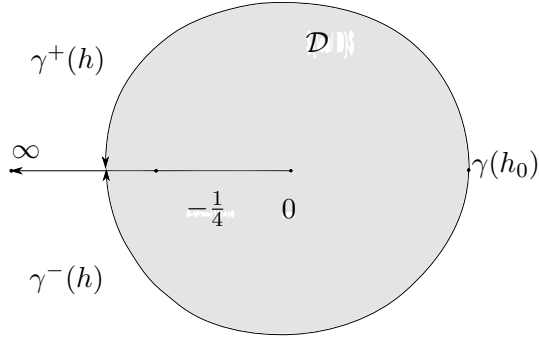
\end{center}
\caption{The analytic continuation of a cycle $\gamma(h)$ in the domain ${\cal D}= \C \setminus [0,-\infty)$}
\label{figgamma}
\end{figure}



\begin{lem}
The first non-vanishing Poincar\'e-Pontryagin-Melnikov function (\ref{first})
has at most two zeros in the complex domain ${\cal D}$.
\label{z1}
\end{lem}
\begin{lem}
The second Poincar\'e-Pontryagin-Melnikov function (\ref{second}) of the  first return map  has at most four zeros in the complex domain ${\cal D}$.
\label{z3}
\end{lem}
\begin{lem}
The Abelian integrals $I_0(h)$ and $I_0'(h)$ do not vanish in ${\cal D}$.
\label{I0}
\end{lem}
\begin{proof}[Proof of Lemma \ref{I0}]
$I_0'(h)$ is a period of the holomorphic one-form $\frac{dx}{y}$ on the elliptic curve $\Gamma_h$, and therefore does not vanish. For real values of $h$ $ I_0'(h)$ represents the period of the orbit $\gamma(h)$ of $d H=0$, while
$I_0(h)$ equals the area of the interior of $\gamma(h)$.
It is remarkable, that $I_0(h)$ does not vanish in a complex domain too.
Indeed, consider  the analytic function
 $$F(h)= \frac{I_0(h)}{I'_0(h)} , h \in \cal D .$$
 We shall count its zeros in ${\cal D}$ by making use of the argument principle as the  proof of previous lemma (see subsection \ref{zerosint}).
\begin{quote}
{\it Let $D\subset\C$ be a relatively compact domain, with a piece-wise smooth boundary.
We suppose, that $f:D\rightarrow\C$ is a continuous function, which is complex-analytic in $D$, except at a finite number of points on the border $\partial D$. We suppose also that $f$ does not vanish on $\partial D$.
Denote by $Z_{D}(f)$ the number of zeros of $f$ in $D$, counted with  multiplicity.
The increment of the argument
\noindent $Var_{\partial D}(argf)$ of $f$ along $\partial D$ oriented counter-clockwise is well defined
and equals the winding number of the curve $f(\partial D)\subset \C$ about the origin, divided by $2\pi$. The argument principle states then that }
\begin{eqnarray}\label{10}
2\pi Z_{D}(f)&=&Var_{\partial D}(argf)
\end{eqnarray}
\end{quote}
Apply now the formula(\ref{10}) to the function $F$ in the intersection of a big disc with a radius $R$ and the complex domain ${\cal D}$. Along the  circle of radius $R$, for $R$ sufficiently big, the decrease of the argument of $F$ is close to $2\pi$, while along the branch cut $(-\infty, 0)$ we have
$$
2 \sqrt{-1} Im(F(h) = F^+(h)-F^-(h)=
\frac{I_{0}(h)}{I'_{0}(h)}-\frac{\overline{I_{0}(h)}}{\overline{I'_{0}(h)}}$$
$$
=\frac{\oint_{\gamma^+}ydx}{\oint_{\gamma^+}\frac{dx}{y}}-\frac{\oint_{\gamma^{-}}ydx}{\oint_{\gamma^{-}}\frac{dx}{y}}=
\frac{W(h)}{|\oint_{\gamma^+}\frac{dx}{y}|^2} .
$$
where
$$
W(h)=det\left(\begin{array}{cc} \oint_{\gamma^+}ydx & \oint_{\gamma^+}\frac{dx}{y}\\ & \\ \oint_{\gamma^-}ydx & \oint_{\gamma^-}\frac{dx}{y} \end{array}\right) .
$$
According to subsection \ref{monodromyext}, the function has two different determinations along $(-\infty,-1/4)$ and $(-1/4,0)$, both of which have no monodromy, and hence are rational in $h$. In fact, (\ref{area}) implies that $W(h)$ is a non-zero constant. If $W(h)=c$ in $(-\infty,-1/4)$, then it equals $2c$ in $(-1/4,0)$. Therefore along the branch cut
the argument of $F^+$ or $F^-$ increases by at most $\pi$.
Summing up the above information, we conclude that $F$ has no zeros in ${\cal D}$.
\end{proof}
\begin{proof}[Proof of Lemma \ref{z1}]
We denote
$$
F(h) = \frac{M_1(h)}{I_0(h)} = M_{1}(h)= \alpha_{0}(h)+\alpha_{2}\frac{I_{2}(h)}{I_{0}(h)}, h\in \cal D
$$

We apply, as in the proof of Lemma \ref{I0}, the argument principle to $F$. Along a big circle the increase of the  argument of $F$ is close to $\pi$. Along the branch cut $(-\infty,0]$ we have
$$
2 \sqrt{-1} Im(F(h)) = F^+(h)-F^-(h)=  \alpha_2 \frac{W(h)}{|I_0(h)|^2}
$$
where
$$
W(h)=
\det\left(\begin{array}{cc} \oint_{\gamma^+}yx^2dx & \oint_{\gamma^+} ydx \\ & \\ \oint_{\gamma^-}yx^2dx & \oint_{\gamma^-}ydx \end{array}\right) = c h (4h+1), c= const.\neq 0 .
$$
Therefore the imaginary part of $F(h)$ along the branch cut $(-\infty,0)$ vanishes at most once, at $-1/4$
Summing up the above information, we get that $F$ has at most two zeros in the complex domain ${\cal D}$.
\end{proof}
\begin{proof}[Proof of Lemma \ref{z3}]
We denote
$$
F(h)=(4h+1)\frac{M_{2}(h)}{I_0(h)}, h\in \cal D,
$$

and apply, as in the proof of Lemma \ref{I0}, the argument principle to $F$.  By making use of (\ref{second}) we have
\begin{equation}
F(h)=\mu(h)\frac{I_{2}(h)}{I_{0}(h)}+\lambda(h)
\label{fh}
\end{equation}
where\\
\begin{equation}
 \lambda(h)=\alpha_{0}+4\alpha_{1}h+4\alpha_{2}h^{2}, \; \mu(h)=\beta_{0}+4\beta_{1}h .
\label{albe}
\end{equation}
Along a big circle the increase of the  argument of $F$ is close to $4\pi$. Along the branch cut $(-\infty,0]$ we have as before
$$
2 \sqrt{-1} Im(F(h)) = F^+(h)-F^-(h)=  \mu(h) \frac{W(h)}{|I_0(h)|^2}
$$
where
$$
W(h)=
\det\left(\begin{array}{cc} \oint_{\gamma^+}yx^2dx & \oint_{\gamma^+} ydx \\ & \\ \oint_{\gamma^-}yx^2dx & \oint_{\gamma^-}ydx \end{array}\right) = c h (4h+1), c= const.\neq 0 .
$$
Therefore the imaginary part of $F(h)$ along the branch cut $(-\infty,0)$ vanishes at most two, at $-1/4$ and at the root of $\mu(h)$.
Summing up the above information, we get that $F$ has at most four zeros in the complex domain ${\cal D}$.
\end{proof}

\section*{Acknowledgements}
We are grateful to the referee for the valuable remarks and corrections. The text was written while the first two authors were visiting the Institute of Mathematics of Toulouse. There are obliged for the hospitality.

\end{document}